\documentclass[12pt]{amsart}

\usepackage{latexsym,amsmath,amssymb,amsthm}
\usepackage{graphicx,comment}
\usepackage[mathscr]{eucal}

\newcommand{\wh}{\widehat}                  

\newtheorem*{thmMain}{Main Theorem}
\newtheorem*{corMain}{Corollary for Knots}

\newtheorem{thm}{Theorem}[section]
\newtheorem{lem}[thm]{Lemma}
\newtheorem{prop}[thm]{Proposition}
\newtheorem{cor}[thm]{Corollary}

\theoremstyle{definition}
\newtheorem{defn}[thm]{Definition}
\newtheorem{exam}[thm]{Example}
\theoremstyle{remark}
\newtheorem{rem}[thm]{Remark}
\newtheorem*{ack}{Acknowledgements}

\begin{document}

\title     [
                Conway potential function for colored links
           ]
           {
                On Conway's potential function
\\
                for colored links
           }
\author    {
                             Boju Jiang
           }
\address   {
                      Department of Mathematics\\
                      Peking University\\
                      Beijing  100871\\
                      China
           }
\email     {
                      bjjiang@math.pku.edu.cn
           }
\thanks    {Partially supported by NSFC grant \#11131008}



\date{}

\begin{abstract}
The Conway potential function (CPF) for colored links is a convenient version
of the multi-variable Alexander-Conway polynomial.
We give a skein characterization of CPF, much simpler than the one by Murakami.
In particular, Conway's `smoothing of crossings' is \emph{not\/} in the axioms.
The proof uses a reduction scheme in a twisted group-algebra $\mathbb P_nB_n$,
where $B_n$ is a braid group and $\mathbb P_n$ is a domain of multi-variable rational fractions.
The proof does not use computer algebra tools.
An interesting by-product is a characterization of the Alexander-Conway polynomial of knots.
\end{abstract}

\maketitle

\section {Introduction}

A \emph{colored link\/} is an oriented link $L=L_1\cup\dots\cup L_\mu$ in $S^3$ together with
a map $\{1,\dots,\mu\}\to\mathbf T$ assigning to each component $L_i$ a color label $t_i\in\mathbf T$.
Here $\mathbf T$ denotes the set of color symbols which will also be regarded
as independent variables in polynomials and rational fractions.
Different components of a link are allowed to share a color.
Two colored links $L$ and $L'$ are \emph{isotopic\/} if there exists an ambient isotopy
from $L$ to $L'$ which respects the color and orientation of each component.

The Alexander polynomial $\Delta_L(t_1,\dots,t_n) \in \mathbb Z[\mathbf T^{\pm1}]$,
introduced by Alexander \cite{A1} in 1928, is an invariant for colored links with $n$-colors.
Here $\mathbb Z[\mathbf T^{\pm1}]$ is the ring of Laurent polynomials with variables in $\mathbf T$, and integer coefficients.
The Alexander polynomial is defined only up to multiplication by a monomial.

The Conway potential function (CPF) of a colored link $L$
is a \emph{well-defined} rational fraction $\nabla_L(t_1,\dots,t_n)$ with variables in $\mathbf T$
that is related to the Alexander polynomial in the following way:
\[
\nabla_L(t_1,\dots,t_n)=
\begin{cases}
\displaystyle\frac{\Delta_L(t^2)}{t-t^{-1}}, & \text{if $n=1$};
\\
\Delta_L(t_1^2,\dots,t_n^2), & \text{if $n>1$}.
\end{cases}
\]
This equality is used to remove the ambiguity in the Alexander polynomial,
to be called the Alexander-Conway polynomial.
CPF was introduced by Conway \cite{C2} in 1970,
but without an explicit model until 1983 by Hartley \cite{H1}.
When two colors $t_i$ and $t_j$ merge,
simply identify the variables $t_i$ and $t_j$ in $\nabla_L(t_1,\dots,t_n)$.

Our main result is

\begin{thmMain}
The Conway potential function $\nabla_L$ is
the invariant of colored links
determined uniquely by the following five axioms,
where $t_i,t_j,t_k$ are arbitrary color symbols, repetition allowed
(the labels $i,j,k$ in diagrams are shorthand for $t_i,t_j,t_k$).
{\allowdisplaybreaks
\begin{gather*}
\nabla
\begin{pmatrix}
\includegraphics[width=.8cm]{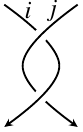}
\end{pmatrix}
+\nabla
\begin{pmatrix}
\includegraphics[width=.8cm]{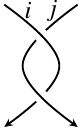}
\end{pmatrix}
=(t_it_j+t_i^{-1}t_j^{-1})
\cdot\nabla
\begin{pmatrix}
\includegraphics[width=.8cm]{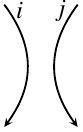}
\end{pmatrix} ;
\tag*{\rm(II)}
\\
\begin{align*}
&(t_i^{-1}t_j^{-1}-t_it_j)
\left\{\nabla
\begin{pmatrix}
\includegraphics[width=1.2cm]{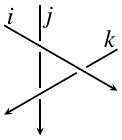}
\end{pmatrix}
+\nabla
\begin{pmatrix}
\includegraphics[width=1.2cm]{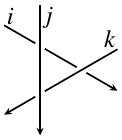}
\end{pmatrix}
\right\}
\tag*{\rm(III)}
\\
&+(t_jt_k-t_j^{-1}t_k^{-1})
\left\{\nabla
\begin{pmatrix}
\includegraphics[width=1.2cm]{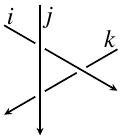}
\end{pmatrix}
+\nabla
\begin{pmatrix}
\includegraphics[width=1.2cm]{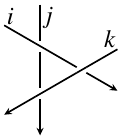}
\end{pmatrix}
\right\}
\\
&+(t_it_k^{-1}-t_i^{-1}t_k)
\left\{\nabla
\begin{pmatrix}
\includegraphics[width=1.2cm]{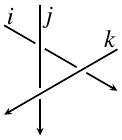}
\end{pmatrix}
+\nabla
\begin{pmatrix}
\includegraphics[width=1.2cm]{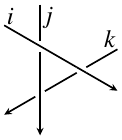}
\end{pmatrix}
\right\}
 =0 ;
\end{align*}
\\
\nabla
\begin{pmatrix}
\includegraphics[height=1.2cm]{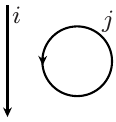}
\end{pmatrix}
=0 ;
\tag*{\rm(IO)}
\\
\nabla
\begin{pmatrix}
\includegraphics[height=1.2cm]{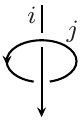}
\end{pmatrix}
=(t_i-t_i^{-1})\cdot
\nabla
\begin{pmatrix}
\;\;
\includegraphics[height=1.2cm]{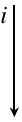}
\;\;
\end{pmatrix} ;
\tag*{($\Phi$)}
\\
\nabla
\begin{pmatrix}
\includegraphics[width=1.35cm]{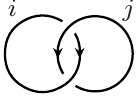}
\end{pmatrix}
=1
\tag*{\rm(H)} .
\end{gather*}
}
\end{thmMain}

This is a characterization of the Conway potential function by skein relations.
It provides a pedestrian's approach to the main conclusions of
the paper \cite{M2} which is hard to read.
Murakami's state model for CPF (Theorem~5.3 of that paper) can be validated
by checking our axioms in a straightforward way.

Cimasoni's model of CPF \cite[Theorem]{C1} can also be validated by our axioms
instead of Murakami's.
In fact, it was this nice geometric model that first led us to
local relations such as (III) and (III$_8$) (see below).

A comparison between our five axioms and Murakami's six axioms \cite[page~126]{M2}:

(A) Conway's ``smoothing of same-color crossings" relation
\[
\nabla
\begin{pmatrix}
\includegraphics[width=.8cm]{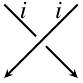}
\end{pmatrix}
-\nabla
\begin{pmatrix}
\includegraphics[width=.8cm]{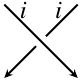}
\end{pmatrix}
=(t_i-t_i^{-1})
\cdot\nabla
\begin{pmatrix}
\includegraphics[width=.8cm]{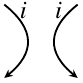}
\end{pmatrix}
\tag*{\rm(I)}
\]
is used by Murakami as his Axiom (1), but is absent in our characterization.

(B) We emphasize that repetition of color labels are allowed in our axioms,
because we have no other means to control same-color crossings.

(C) Our three-string Axiom (III) is simpler than Murakami's Axiom (3),
see Example~\ref{exam:(III_7)} below.

(D) Our Axioms (II), (IO) and ($\Phi$) are the same as Murakami's Axioms (2), (6) and (5), respectively.
In the presence of Axiom ($\Phi$), our Axiom (H) (normalization at the Hopf link) is equivalent to
Murakami's Axiom (4) (normalization at the trivial knot).
So this difference is minor.

The point (A) above is remarkable.
As its consequence, we have
\begin{corMain}
\label{cor:knots}
The Alexander-Conway polynomial $\Delta_K \in \mathbb Z[t^{\pm1}]$
for knots is the invariant of knots determined uniquely by the following three axioms.
{\allowdisplaybreaks
\begin{gather*}
\Delta
\begin{pmatrix}
\includegraphics[width=.8cm]{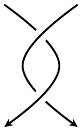}
\end{pmatrix}
+\Delta
\begin{pmatrix}
\includegraphics[width=.8cm]{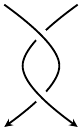}
\end{pmatrix}
=(t+t^{-1})
\cdot\Delta
\begin{pmatrix}
\includegraphics[width=.8cm]{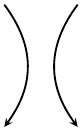}
\end{pmatrix} ;
\tag*{\rm(II$_\text{K}$)}
\\
\Delta
\begin{pmatrix}
\includegraphics[width=1.2cm]{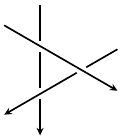}
\end{pmatrix}
+\Delta
\begin{pmatrix}
\includegraphics[width=1.2cm]{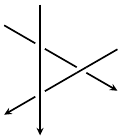}
\end{pmatrix}
=
\Delta
\begin{pmatrix}
\includegraphics[width=1.2cm]{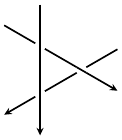}
\end{pmatrix}
+\Delta
\begin{pmatrix}
\includegraphics[width=1.2cm]{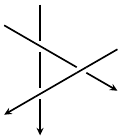}
\end{pmatrix} ;
\tag*{\rm(III$_\text{K}$)}
\\
\Delta
\begin{pmatrix}
\includegraphics[width=.8cm]{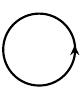}
\end{pmatrix}
=1 .
\tag*{\rm(O)}
\end{gather*}
}
\end{corMain}

Note that the relation (III$_\text{K}$) is different from Conway's four-term relation
(cf.\ Corollary~\ref{cor:relation(III_4)&(III_8)}).

It is well known that the uncolored Alexander-Conway polynomial is
characterized by the classical Conway relation
\[
\Delta
\begin{pmatrix}
\includegraphics[width=.8cm]{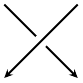}
\end{pmatrix}
-\Delta
\begin{pmatrix}
\includegraphics[width=.8cm]{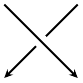}
\end{pmatrix}
=(t^{\frac12}-t^{-\frac12})
\cdot\Delta
\begin{pmatrix}
\includegraphics[width=.8cm]{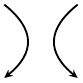}
\end{pmatrix}
\tag*{\rm(C)}
\]
and (O). The two sides of (C) always have different number of components.
The point here is that we now have a characterization within the realm of knots.

The structure of the paper is as follows.
New skein relations are gathered in Section~\ref{sec:skeins}.
The proof in Section~\ref{sec:proof_skein} is based on the original model of CPF in Hartley~\cite{H1}.
In Section~\ref{sec:braids} we develop the language of colored braids
and interpret skein relations in the twisted group algebra $\mathbb P_nB_n$
over a domain $\mathbb P_n$ of multi-variable rational fractions.
An algebraic reduction lemma in $\mathbb P_nB_n$ in Section~\ref{sec:key_lem}
is the key to the proof of our main results in Section~\ref{sec:proof_Main}.

\begin{ack}
The author thanks Hao Zheng for helpful discussions.
He also thanks the referee for questions that prompted clarification of several points in the exposition.
\end{ack}

\section {Some `triangular' skein relations for CPF}
\label{sec:skeins}

In the theorem below we consider link projections that are identical
except within a disk where they differ in a specific way.
The proof is postponed to the next section, after a discussion of corollaries.

\begin{thm}
\label{thm:relation(III_6)}
The Conway potential function satisfies the following skein relation,
where $t_i,t_j,t_k$ are arbitrary color symbols, repetition allowed,
$i,j,k$ are just shorthand for them in diagrams.
{\allowdisplaybreaks
\begin{align*}
&(t_i^{-1}t_j^{-1}-t_it_j)
\left\{\nabla
\begin{pmatrix}
\includegraphics[width=1.2cm]{fig_3icons/1b2b1a_ddd}
\end{pmatrix}
+\nabla
\begin{pmatrix}
\includegraphics[width=1.2cm]{fig_3icons/1a2a1b_ddd}
\end{pmatrix}
\right\}
\tag*{\rm(III)}
\\
&+(t_jt_k-t_j^{-1}t_k^{-1})
\left\{\nabla
\begin{pmatrix}
\includegraphics[width=1.2cm]{fig_3icons/1a2b1b_ddd}
\end{pmatrix}
+\nabla
\begin{pmatrix}
\includegraphics[width=1.2cm]{fig_3icons/1b2a1a_ddd}
\end{pmatrix}
\right\}
\\
&+(t_it_k^{-1}-t_i^{-1}t_k)
\left\{\nabla
\begin{pmatrix}
\includegraphics[width=1.2cm]{fig_3icons/1a2a1a_ddd}
\end{pmatrix}
+\nabla
\begin{pmatrix}
\includegraphics[width=1.2cm]{fig_3icons/1b2b1b_ddd}
\end{pmatrix}
\right\}
 =0.
\end{align*}
}
\end{thm}

The relation (III) actually has $12$ terms when all brackets are expanded.
In the following corollary, (III$_4$) is Conway's four-term relation which was not given a proof in \cite{H1}.
The eight-term relation (III$_8$) is new.
\begin{cor}
\label{cor:relation(III_4)&(III_8)}
The Conway potential function satisfies the following skein relations:
{\allowdisplaybreaks
\begin{gather*}
\begin{aligned}
\nabla
\begin{pmatrix}
\includegraphics[width=1.2cm]{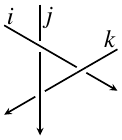}
\end{pmatrix}
+\nabla
\begin{pmatrix}
\includegraphics[width=1.2cm]{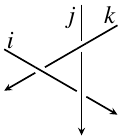}
\end{pmatrix}
&=
\nabla
\begin{pmatrix}
\includegraphics[width=1.2cm]{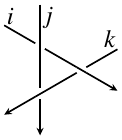}
\end{pmatrix}
+\nabla
\begin{pmatrix}
\includegraphics[width=1.2cm]{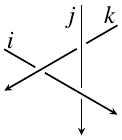}
\end{pmatrix}
,
\end{aligned}
\tag*{(III$_4$)}
\\
\begin{aligned}[t]
& t_it_j\cdot\nabla
\begin{pmatrix}
\includegraphics[width=1.2cm]{fig_3icons/1b2b1a_ddd}
\end{pmatrix}
-t_i^{-1}t_j^{-1}\cdot\nabla
\begin{pmatrix}
\includegraphics[width=1.2cm]{fig_3icons/1a2a1b_ddd}
\end{pmatrix}
\\
&+t_j^{-1}t_k^{-1}\cdot\nabla
\begin{pmatrix}
\includegraphics[width=1.2cm]{fig_3icons/1a2b1b_ddd}
\end{pmatrix}
-t_jt_k\cdot\nabla
\begin{pmatrix}
\includegraphics[width=1.2cm]{fig_3icons/1b2a1a_ddd}
\end{pmatrix}
\\
&+t_i^{-1}t_k\cdot\nabla
\begin{pmatrix}
\includegraphics[width=1.2cm]{fig_3icons/1a2a1a_ddd}
\end{pmatrix}
-t_it_k^{-1}\cdot\nabla
\begin{pmatrix}
\includegraphics[width=1.2cm]{fig_3icons/1b2b1b_ddd}
\end{pmatrix}
\\
&+\nabla
\begin{pmatrix}
\includegraphics[width=1.2cm]{fig_3icons/1b2a1b_ddd}
\end{pmatrix}
-\nabla
\begin{pmatrix}
\includegraphics[width=1.2cm]{fig_3icons/1a2b1a_ddd}
\end{pmatrix}
 =0.
\end{aligned}
\tag*{\rm(III$_\text{8}$)}
\end{gather*}
}
\end{cor}

\begin{proof}
See Example~\ref{exam:relation(III_4)&(III_8)}
as an application of the skein relator ideal.
\end{proof}

\begin{rem}
The relations (III$_4$) and (III$_8$), jointly, imply (III),
as shown in Example~\ref{exam:relation(III_4)&(III_8)}.
We do not know whether Conway's four term relation (III$_4$), or the eight term relation (III$_8$), alone,
is powerful enough to replace Axiom (III) in the Main Theorem.
\end{rem}

\section {Proof of Theorem~\ref{thm:relation(III_6)}}
\label{sec:proof_skein}

In this section we assume the reader is familiar with the
terminology and notation in Sections 2 and 3 of \cite{H1}.
Since we use indices $i,j,k$ for colors,
we shall use indices $a,b,c$ for crossing points/generating arcs in a link projection,
as well as for rows/columns in a matrix.

Let $K_r$, $r=1,\dots,6$, denote the six colored links appearing in the relation (III), in that order.
Number the crossing points $P_1,\dots,P_m$,
then number the generating arcs so that $u_a$ is the generating arc exiting from $P_a$.
Outside of the depicted region, the crossing points and generating arcs
should be numbered and colored in the same way for all $K_r$.

Colors are marked by symbols $\{t_1,\dots,t_n\}$.
A coloring of a link is a map $\theta:\{u_1,\dots,u_m\}\to\{t_1,\dots,t_n\}$
which takes $u_a$ to $t_i$ if the generating arc $u_a$ has the $i$-th color.
We use the same notation for the induced linear map
$\theta: \mathbb Z[u_1^{\pm1},\dots,u_m^{\pm1}] \to \mathbb Z[t_1^{\pm1},\dots,t_n^{\pm1}]$.

Focus on any one link $K_r$.
Around each crossing point $P_a$ consider a small anticlockwise circle,
starting at a point to the right of both over- and under-crossing arcs at $P_a$.
This gives a Wirtinger relator $R_a$, a word in the free group with basis $\{u_1,\dots,u_m\}$.
Let $M_r$ be the $m\times m$ Jacobian matrix (in the sense of free differential calculus)
$M_r:=(M_{ab})$ with entries $M_{ab}:=(\partial R_a/\partial u_b)^\theta$.
This $M_r$ is a matrix in $\mathbb Z[t_1^{\pm1},\dots,t_n^{\pm1}]$.

Let $M_r^{(ab)}$ denote the matrix obtained from $M_r$
by deleting the $a$-th row and the $b$-th column.
Then the rational fraction
\[
D_r(t_1,\dots,t_n):=(-1)^{a+b}\det(M_r^{(ab)})/\mathfrak w_{a}^\theta(u_{b}^\theta-1)
\tag* {\cite{H1}(2.3)},
\]
which lies in the quotient field $\mathbb Q(t_1,\dots,t_n)$ of $\mathbb Z[t_1^{\pm1},\dots,t_n^{\pm1}]$,
is independent of the choice of $a,b$.

For each $i$, let $\kappa_i$ denote the total ``curvature'' of components
of $K_r$ with $i$-th color.
Let $\nu'_i$ and $\nu''_{i,r}$ be the number of crossing points,
\emph{outside and inside of the depicted region, respectively, }
at which the overcrossing arc has $i$-th color in the link $K_r$.
Let $\mu_{i,r}:=\kappa_i-\nu'_i-\nu''_{i,r}$.
Let $\phi:\mathbb Q(t_1,\dots,t_n)\to\mathbb Q(t_1,\dots,t_n)$ be
the substitution map sending each $t_i$ to $t_i^2$.
Then the Conway potential function of $K_r$ is
\[
\nabla_r(t_1,\dots,t_n) = D_r(t_1,\dots,t_n)^\phi
\cdot t_1^{\mu_{1,r}} \cdots t_n^{\mu_{n,r}} .
\tag* {\cite{H1}(2.4)}
\]

In Figure 1 we give pictures in the style of \cite{H1},
where $t_i,t_j,t_k$ (repetition allowed) are color labels of the strings $u_1,u_2,u_3$, respectively.
In each picture we add an anticlockwise curl as shown,
increasing the number of crossing points to $m+1$.
Clearly, adding such a curl does not change $\mu_{i,r}$ and $\det(M_r^{(ab)})$.
We can also assume $a,b>6$ (by adding more curls elsewhere if necessary).

\begin{figure}
\includegraphics[width=2.5cm]{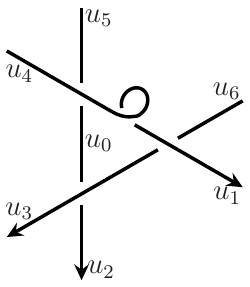}
\qquad
\includegraphics[width=2.5cm]{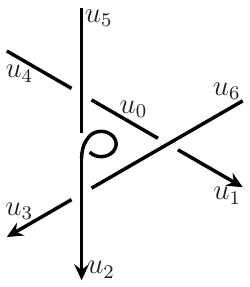}
\qquad
\includegraphics[width=2.5cm]{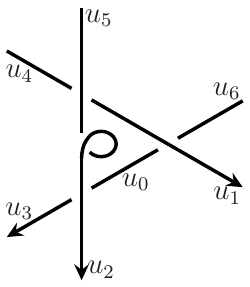}
\qquad
\includegraphics[width=2.5cm]{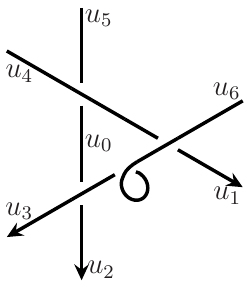}
\\
$K_1$ \hspace{2.67cm} $K_2$ \hspace{2.67cm} $K_3$ \hspace{2.67cm} $K_4$
\\[3ex]
\includegraphics[width=2.5cm]{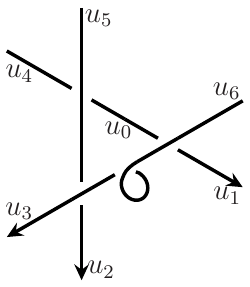}
\qquad
\includegraphics[width=2.5cm]{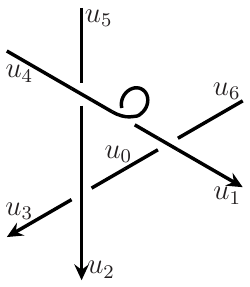}
\\
$K_5$ \hspace{2.67cm} $K_6$
\\[2ex]
{\sc Figure 1}
\end{figure}

In the following computation, the column vectors $\mathbf c_1, \dots, \mathbf c_6$
and the $(m-4)\times(m-7)$ matrix $\mathbf *$ are common to all $r$.
By Laplace expansion along the first four rows,
$\det(M_r^{(ab)})$ is a linear combination of $(m-4)\times(m-4)$ minors
$S_{\alpha\beta\gamma}:=
\|
\begin{array}{ccc|c}
\mathbf c_\alpha
& \mathbf c_\beta
& \mathbf c_\gamma
& \mathbf *
\end{array}
\|$
for $1\leq\alpha<\beta<\gamma\leq6$.
The coefficients are $4\times4$ minors, readily read off from the matrix $M_r^{(ab)}$.
See Table~\ref{table:braidform}.
We will use the shorthand $\delta_r:=\det(M_r^{(ab)})$.

For $K_1$, we have Wirtinger relators and Jacobian
{\allowdisplaybreaks
\begin{gather*}
R_0=u_0u_4u_5^{-1}u_4^{-1}, \quad
R_1=u_1u_4^{-1}, \quad
R_2=u_3u_2u_3^{-1}u_0^{-1}, \quad
R_3=u_3u_1u_6^{-1}u_1^{-1},
\\
\det(M_1^{(ab)})
=
\left\|
\begin{array}{ccccccc|c}
1 & 0 & 0 & 0 & t_j-1 & -t_i & 0 & \mathbf 0
\\
0 & 1 & 0 & 0 & -1 & 0 & 0 & \mathbf 0
\\
-1 & 0 & t_k & 1-t_j & 0 & 0 & 0 & \mathbf 0
\\
0 & t_k-1 & 0 & 1 & 0 & 0 & -t_i & \mathbf 0
\\
\hline
\mathbf 0 & \mathbf c_1 & \mathbf c_2 & \mathbf c_3 & \mathbf c_4 & \mathbf c_5 & \mathbf c_6 & \mathbf *
\end{array}
\right\| .
\end{gather*}
} 

For $K_2$, we have Wirtinger relators
{\allowdisplaybreaks
\begin{gather*}
R_0=u_5u_0u_5^{-1}u_4^{-1}, \quad
R_1=u_6u_1u_6^{-1}u_0^{-1}, \quad
R_2=u_2u_5^{-1}, \quad
R_3=u_3u_2u_6^{-1}u_2^{-1},
\\
\det(M_2^{(ab)})
=
\left\|
\begin{array}{ccccccc|c}
t_j & 0 & 0 & 0 & -1 & 1-t_i & 0 & \mathbf 0
\\
-1 & t_k & 0 & 0 & 0 & 0 & 1-t_i & \mathbf 0
\\
0 & 0 & 1 & 0 & 0 & -1 & 0 & \mathbf 0
\\
0 & 0 & t_k-1 & 1 & 0 & 0 & -t_j & \mathbf 0
\\
\hline
\mathbf 0 & \mathbf c_1 & \mathbf c_2 & \mathbf c_3 & \mathbf c_4 & \mathbf c_5 & \mathbf c_6 & \mathbf *
\end{array}
\right\| .
\end{gather*}
} 

{\allowdisplaybreaks
For $K_3$, we have Wirtinger relators
\begin{gather*}
R_0=u_0u_1u_6^{-1}u_1^{-1}, \quad
R_1=u_5u_1u_5^{-1}u_4^{-1}, \quad
R_2=u_2u_5^{-1}, \quad
R_3=u_3u_2u_0^{-1}u_2^{-1},
\\
\det(M_3^{(ab)})
=
\left\|
\begin{array}{ccccccc|c}
1 & t_k-1 & 0 & 0 & 0 & 0 & -t_i & \mathbf 0
\\
0 & t_j & 0 & 0 & -1 & 1-t_i & 0 & \mathbf 0
\\
0 & 0 & 1 & 0 & 0 & -1 & 0 & \mathbf 0
\\
-t_j & 0 & t_k-1 & 1 & 0 & 0 & 0 & \mathbf 0
\\
\hline
\mathbf 0 & \mathbf c_1 & \mathbf c_2 & \mathbf c_3 & \mathbf c_4 & \mathbf c_5 & \mathbf c_6 & \mathbf *
\end{array}
\right\| .
\end{gather*}
} 

For $K_4$, we have Wirtinger relators
{\allowdisplaybreaks
\begin{gather*}
R_0=u_0u_4u_5^{-1}u_4^{-1}, \quad
R_1=u_6u_1u_6^{-1}u_4^{-1}, \quad
R_2=u_3u_2u_3^{-1}u_0^{-1}, \quad
R_3=u_3u_6^{-1},
\\
\det(M_4^{(ab)})
=
\left\|
\begin{array}{ccccccc|c}
1 & 0 & 0 & 0 & t_j-1 & -t_i & 0 & \mathbf 0
\\
0 & t_k & 0 & 0 & -1 & 0 & 1-t_i & \mathbf 0
\\
-1 & 0 & t_k & 1-t_j & 0 & 0 & 0 & \mathbf 0
\\
0 & 0 & 0 & 1 & 0 & 0 & -1 & \mathbf 0
\\
\hline
\mathbf 0 & \mathbf c_1 & \mathbf c_2 & \mathbf c_3 & \mathbf c_4 & \mathbf c_5 & \mathbf c_6 & \mathbf *
\end{array}
\right\| .
\end{gather*}
} 

For $K_5$, we have Wirtinger relators
{\allowdisplaybreaks
\begin{gather*}
R_0=u_5u_0u_5^{-1}u_4^{-1}, \quad
R_1=u_6u_1u_6^{-1}u_0^{-1}, \quad
R_2=u_3u_2u_3^{-1}u_5^{-1}, \quad
R_3=u_3u_6^{-1},
\\
\det(M_5^{(ab)})
=
\left\|
\begin{array}{ccccccc|c}
t_j & 0 & 0 & 0 & -1 & 1-t_i & 0 & \mathbf 0
\\
-1 & t_k & 0 & 0 & 0 & 0 & 1-t_i & \mathbf 0
\\
0 & 0 & t_k & 1-t_j & 0 & -1 & 0 & \mathbf 0
\\
0 & 0 & 0 & 1 & 0 & 0 & -1 & \mathbf 0
\\
\hline
\mathbf 0 & \mathbf c_1 & \mathbf c_2 & \mathbf c_3 & \mathbf c_4 & \mathbf c_5 & \mathbf c_6 & \mathbf *
\end{array}
\right\| .
\end{gather*}
} 

For $K_6$, we have Wirtinger relators and Jacobian
{\allowdisplaybreaks
\begin{gather*}
R_0=u_0u_1u_6^{-1}u_1^{-1}, \quad
R_1=u_1u_4^{-1}, \quad
R_2=u_2u_4u_5^{-1}u_4^{-1}, \quad
R_3=u_3u_2u_0^{-1}u_2^{-1},
\\
\det(M_6^{(ab)})
=
\left\|
\begin{array}{ccccccc|c}
1 & t_k-1 & 0 & 0 & 0 & 0 & -t_i & \mathbf 0
\\
0 & 1 & 0 & 0 & -1 & 0 & 0 & \mathbf 0
\\
0 & 0 & 1 & 0 & t_j-1 & -t_i & 0 & \mathbf 0
\\
-t_j & 0 & t_k-1 & 1 & 0 & 0 & 0 & \mathbf 0
\\
\hline
\mathbf 0 & \mathbf c_1 & \mathbf c_2 & \mathbf c_3 & \mathbf c_4 & \mathbf c_5 & \mathbf c_6 & \mathbf *
\end{array}
\right\| .
\end{gather*}
} 

\begin{table}[tbp]
\renewcommand{\arraystretch}{1.2} 
\centering
\begin{tabular}{||c|| c | c | c | c | c | c ||}
\hline\hline
 & $\delta_1$ & $\delta_2$ & $\delta_3$ & $\delta_4$ & $\delta_5$ & $\delta_6$
\\ \hline\hline
$S_{123}$
& $t_i^2$ & $t_j$ & $t_it_j$ & $t_i$ & $1$ & $t_i^2t_j$
\\ \hline
$S_{124}$
& $0$ & $t_j(1-t_i)$ & $0$
& $t_i(1-t_i)$ & $1-t_i$ & $0$
\\ \hline
$S_{125}$
& $t_i(1-t_j)$ & $0$ & $0$
& $t_i(1-t_j)$ & $1-t_j$ & $0$
\\ \hline
$S_{126}$
& $t_i$ & $1$ & $1$ & $t_i$ & $1$ & $t_i$
\\ \hline
$S_{134}$
& $0$ & $t_jt_k(t_i-1)$
& $t_it_j(t_i-1)$ & $0$
& $t_k(t_i-1)$ & $0$
\\ \hline
$S_{135}$
& $-t_it_k$ & $-t_j$ & $-t_it_j$ & $-t_k$ & $-t_k$ & $-t_it_j$
\\ \hline
$S_{136}$
& $0$ & $1-t_k$ & $1-t_k$ & $0$ & $0$
& $t_i(1-t_k)$
\\ \hline
$S_{145}$
& $0$ & $t_j(t_i-1)$ & $0$
& $t_k(t_i-1)$
& $t_jt_k(t_i-1)$ & $0$
\\ \hline
$S_{146}$
& $0$ & $1-t_i$ & $1-t_i$ & $0$
& $t_k(1-t_i)$ & $0$
\\ \hline
$S_{156}$
& $t_k$ & $1$ & $1$ & $t_k$ & $t_k$ & $1$
\\ \hline
$S_{234}$
& $t_i^2$ & $t_j^2t_k$ & $t_it_j^2$ & $t_it_k$ & $t_jt_k$ & $t_i^2t_j$
\\ \hline
$S_{235}$
& $t_i(t_j-1)$ & $0$ & $0$
& $t_k(t_j-1)$ & $0$
& $t_it_j(t_j-1)$
\\ \hline
$S_{236}$
& $t_i(t_k-1)$ & $0$
& $t_j(t_k-1)$ & $0$ & $0$
& $t_it_j(t_k-1)$
\\ \hline
$S_{245}$
& $t_i(t_j-1)$ & $0$ & $0$
& $t_k(t_j-1)$
& $t_jt_k(t_j-1)$ & $0$
\\ \hline
$S_{246}$
& $-t_i$ & $-t_jt_k$ & $-t_j$ & $-t_it_k$ & $-t_jt_k$ & $-t_i$
\\ \hline
$S_{256}$
& $t_k(1-t_j)$ & $0$ & $0$
& $t_k(1-t_j)$ & $0$ & $1-t_j$
\\ \hline
$S_{345}$
& $t_it_k$ & $t_j^2t_k$ & $t_it_j^2$ & $t_k^2$ & $t_jt_k^2$ & $t_it_j$
\\ \hline
$S_{346}$
& $0$ & $t_jt_k(t_k-1)$
& $t_it_j(t_k-1)$
& $0$ & $0$ & $t_i(t_k-1)$
\\ \hline
$S_{356}$
& $t_k(1-t_k)$ & $0$
& $t_j(1-t_k)$ & $0$ & $0$ & $1-t_k$
\\ \hline
$S_{456}$
& $t_k$ & $t_jt_k$ & $t_j$ & $t_k^2$ & $t_jt_k^2$ & $1$
\\ \hline\hline
\end{tabular}
\caption{\protect Coefficient of $S_{\alpha\beta\gamma}$ in $\delta_r:=\det(M_r^{(ab)})$}
\label{table:braidform}
\end{table}

Let
\[
\varepsilon:=t_1^{\kappa_1-\nu'_1} \cdots t_n^{\kappa_n-\nu'_{n}} \cdot
[(-1)^{a+b}/\mathfrak w_{a}^{\theta}(u_{b}^{\theta}-1)]^\phi.
\]
Then
$\nabla_r = \varepsilon \cdot t_1^{-\nu''_{1,r}}\cdots t_n^{-\nu''_{n,r}} \delta_r^\phi$. Namely,
\begin{alignat*}{2}
\nabla_1 &= \varepsilon \cdot t_i^{-2}t_k^{-1} \delta_1^\phi, &\qquad
\nabla_2 &= \varepsilon \cdot t_j^{-2}t_k^{-1} \delta_2^\phi,
\\
\nabla_3 &= \varepsilon \cdot t_i^{-1}t_j^{-2} \delta_3^\phi, &\qquad
\nabla_4 &= \varepsilon \cdot t_i^{-1}t_k^{-2} \delta_4^\phi,
\\
\nabla_5 &= \varepsilon \cdot t_j^{-1}t_k^{-2} \delta_5^\phi, &\qquad
\nabla_6 &= \varepsilon \cdot t_i^{-2}t_j^{-1} \delta_6^\phi.
\end{alignat*}

Denote the left hand side of the relation (III) by $F$. Then
\begin{align*}
t_it_jt_k\cdot F
&=(t_i^2t_j^2t_k-t_k)(\nabla_1+\nabla_2)
+(t_i-t_it_j^2t_k^2)(\nabla_3+\nabla_4)
\\
&\quad +(t_jt_k^2-t_i^2t_j)(\nabla_5+\nabla_6)
\\
&=\varepsilon\cdot \left[
(t_j-t_i^{-1})\delta_1-(t_j^{-1}-t_i)\delta_2+(t_j^{-1}-t_k)\delta_3-(t_j-t_k^{-1})\delta_4 \right.
\\
&\qquad\quad \left. +(1-t_it_k^{-1})\delta_5-(1-t_i^{-1}t_k)\delta_6 \right]^\phi.
\end{align*}
From Table~\ref{table:braidform}, it is straightforward to check
that the expression enclosed in brackets equals to $0$.
Hence $F=0$.
\qed

\section{Colored braids and multi-variable skein relators}
\label{sec:braids}

We first set up a language of colored braids which is more flexible than the one in \cite[Definition 3.2]{M2}.
Then, more importantly, we clarify the meaning of skein relations.
The symbols $t_i$ and $t_j$ in a skein relation, e.g.\ (III) in Theorem~\ref{thm:relation(III_6)},
carry the location information of strings, so a priori to be distinguished.
On the other hand, $t_i$ and $t_j$ are also interpreted as color symbols,
so may be indistinguishable if the relevant strings have the same color.
This dual role may cause confusion when we operate on colored braids.
To resolve this problem we shall make a distinction between $\mathbb P_n$ and $\mathbb P_{\mathbf T}$,
and prefer the notion of skein relators to the commonly used language of skein relations.

\subsection{Links as closed braids}
For braids, we use the following conventions:
Braids are drawn from top to bottom.
The strands of a braid are numbered at the top of the braid, from left to right.
The product $\beta_1\cdot\beta_2$ of two $n$-braids is obtained by drawing $\beta_2$ below $\beta_1$.
The set $B_n$ of all $n$-braids forms a group under this multiplication,
with standard generators $\sigma_1,\sigma_2,\dots,\sigma_{n-1}$.
Each $n$-braid $\beta$ has an \emph{underlying permutation} of $\{1,\dots,n\}$, denoted $i\mapsto i^{\beta}$,
where $i^{\beta}$ is the position of the $i$-th strand at the bottom of $\beta$.
In this way the braid group $B_n$ projects onto the symmetric group $\mathfrak S_n$.

It is well known that links can be presented as closed braids.
The closure of a braid $\beta\in B_n$ will be denoted $\wh\beta$.
Two braids (possibly with different number of strands) have isotopic closures
if and only if they can be related by a finite sequence of two types of moves:
\begin{enumerate}
\item Conjugacy move: $\beta$ $\leftrightsquigarrow$ $\beta'$ where $\beta,\beta'$ are conjugate in a braid group $B_n$;
\item Markov move: $\beta\in B_n$ $\leftrightsquigarrow$ $\beta\sigma_n^{\pm1}\in B_{n+1}$.
\end{enumerate}

\subsection{Colored braids}
\label{subsec:colored_braid}

Let $\mathbf T$ be the set of color symbols.
A \emph{colored $n$-braid\/} $\beta^{(t_1,\dots,t_n)}$ is defined to be an $n$-braid $\beta$
equipped with a sequence $(t_1,\dots,t_n)$ of color symbols (repetition allowed),
where $t_i\in\mathbf T$ is the color label on the $i$-th strand
\emph{at the top} of the braid.
By convention, when we refer to a colored $n$-braid $\beta$ without specifying its color sequence,
it is understood that the color label on the $i$-th strand (at the top) is denoted $t_i$.

Our notion of colored braids is more general than Murakami's (see \cite[Definition 3.2]{M2})
which is the same as the notion of closable colored braid below.
We do \emph{not\/} insist that the color sequences at top and bottom should match.
Colored $n$-braids do not form a group.

\subsection{Colored links as closures of colored braids}

When a colored link $L$ is presented as the closure of an $n$-braid $\beta$,
the coloring of $L$ puts a sequence $(t_1,\dots,t_n)$ of color labels at the top of $\beta$,
and gives us a colored braid $\beta^{(t_1,\dots,t_n)}$,
or $\beta$ for brevity by our convention in Subsection~\ref{subsec:colored_braid}.
In such a case the color sequences at the top and bottom of $\beta$ must match,
that is, $t_i=t_{i^\beta}$ for all $i$.
A colored braid with this property will be called \emph{closable}.
Note that there may be repetitions in the color sequence $(t_1,\dots,t_n)$.
The colored link $L$ is then represented as the closure of the closable colored braid $\beta$,
written $L=\wh\beta$.

\begin{prop}
\label{prop:colored_Markov}
\textup{(\cite[Theorems 3.3 and 3.5]{M2})}
Every colored link can be presented as the closure of a closable colored braid.
Two closable colored braids (possibly with different number of strands) have isotopic closures
if and only if they can be related by a finite sequence of two types of moves:
\begin{enumerate}
\item Conjugacy move: $\beta^{(t_1,\dots,t_n)}$ $\leftrightsquigarrow$ $\beta'{}^{(t'_1,\dots,t'_n)}$,
where the $n$-braids $\beta'=\alpha\beta\alpha^{-1}$ for some $n$-braid $\alpha$,
and the color sequences $t'_i=t_{i^\alpha}$ for all $i$;
\item Markov move: the colored $n$-braid $\beta^{(t_1,\dots,t_n)}$ $\leftrightsquigarrow$
the colored $(n+1)$-braid $(\beta\sigma_n^{\pm1})^{(t_1,\dots,t_n,t_n)}$.
\qed
\end{enumerate}
\end{prop}

Beware of the color sequence specifications in these moves of colored braids.

\subsection{Skein relators}
\label{subsec:skein_relators}

\begin{defn}
\label{defn:persistent_fractions}
Let $\mathbb C[t_1,\dots,t_n]$ denote the polynomial ring in $n$ variables,
and let $\mathbb C(t_1,\dots,t_n)$ be its field of fractions, i.e.,
the field of rational fractions in $n$ variables.
Here the variables $t_1,\dots,t_n$ are independent of each other.
The reducing homomorphism $\rho:\mathbb C[t_1,\dots,t_n]\to\mathbb C[t]$
sends all $t_i$'s into a single variable $t$.
A rational fraction in $\mathbb C(t_1,\dots,t_n)$ will be called \emph{persistent} if,
when it is written in simplified form $f/g$ where $f$ and $g$ are polynomials without common factors,
the denominator $g$ does not vanish under $\rho$.
Let $\mathbb P_n$ denote the subring of $\mathbb C(t_1,\dots,t_n)$ consisting of all persistent fractions.
The above homomorphism $\rho$ extends to a ring homomorphism $\rho:\mathbb P_n\to\mathbb C(t)$
(but is not extendable to $\mathbb C(t_1,\dots,t_n)$).
An element of $\mathbb P_n$ is invertible in $\mathbb P_n$ if and only if it is not killed by $\rho$.
Clearly the Laurent polynomial ring $\mathbb Z[t_1^{\pm1},\dots,t_n^{\pm1}]$ is contained in $\mathbb P_n$.
The integral domain $\mathbb P_n$ is where the coefficients in skein relations belong.

Similarly, for a set $\mathbf T$ of color symbols,
let $\mathbb C[\mathbf T]$ denote the polynomial ring with variables in $\mathbf T$,
and let $\mathbb C(\mathbf T)$ be the field of rational fractions in these variables.
The reducing homomorphism $\rho:\mathbb C[\mathbf T]\to\mathbb C[t]$
sends all of $\mathbf T$ into a single variable $t$.
As in the last paragraph, $\rho$ extends to a homomorphism $\rho:\mathbb P_{\mathbf T}\to\mathbb C(t)$
from the ring of persistent rational fractions with variables in $\mathbf T$.
The integral domain $\mathbb P_{\mathbf T}$ is where the Conway potential functions belong,
and the reducing homomorphism $\rho$ sends the colored CPF to the uncolored CPF.

Let $\mathbb P_n B_n$ be the free $\mathbb P_n$-module with basis $B_n$.
An element of $\mathbb P_nB_n$ is called \emph{homogeneous} if all its terms
(with nonzero coefficients) have the same underlying permutation.
\end{defn}

\begin{defn}
\label{defn:skein_relation}
We say that a homogeneous element
\[
C_1\cdot\beta_1+\dots+C_k\cdot\beta_k 
\]
of $\mathbb P_n B_n$ is a \emph{skein relator}, or equivalently, say that the corresponding formal equation
(in which $\nabla_{L_{\beta_h}}$ stands for the $\nabla$ of the link $L_{\beta_h}$)
\[
C_1\cdot\nabla_{L_{\beta_1}}+\dots+C_k\cdot\nabla_{L_{\beta_k}}=0
\]
is a \emph{skein relation}, if the following condition is satisfied:
For any \emph{colored\/} links $L_{\beta_1},\dots,L_{\beta_k}$ that are identical except in a cylinder
where they are represented by the braids $\beta_1,\dots,\beta_k$ respectively
(consequently these braids inherit a common color sequence at the top of that cylinder),
the formal equation becomes an equality in $\mathbb P_{\mathbf T}$
when the variables in the $C_h$'s are substituted by that sequence of color symbols.
\end{defn}

\begin{exam}
\label{exam:relator_vs_relation}
Suppose a homogeneous element $C_1\cdot\beta_1+\dots+C_k\cdot\beta_k\in \mathbb P_nB_n$ is given.
If the $\beta_h$'s are colored by a common color sequence $t_1,\dots,t_n$ (at the top) which makes them closable,
then by taking closures and replacing the variables with the color sequence,
we get a corresponding rational fraction
\[
C_1\cdot \nabla_{\wh\beta_1}+\dots +C_k\cdot \nabla_{\wh\beta_k} \in \mathbb P_{\mathbf T}.
\]
The latter vanishes if the former is a skein relator.
\end{exam}

\begin{exam}
The skein relations (II) (from \cite[(5.1)]{H1}), (III) (from Theorem~\ref{thm:relation(III_6)}),
(III$_4$) and (III$_8$) (from Corollary~\ref{cor:relation(III_4)&(III_8)})
correspond to the following skein relators.
(The symbol $e$ stands for the trivial braid.)
{\allowdisplaybreaks
\begin{gather*}
\text{\rm(II$_\text{B}$)}:=
{\sigma_1^2} +\sigma_1^{-2} -(t_1t_2+t_1^{-1}t_2^{-1})\cdot{e} ;
\\
\text{\rm(III$_\text{B}$)}:=
\begin{aligned}[t]
&(t_1^{-1}t_2^{-1}-t_1t_2)\cdot
({\sigma_1^{-1}\sigma_2^{-1}\sigma_1} + {\sigma_1\sigma_2\sigma_1^{-1}})
\\
&+ (t_2t_3-t_2^{-1}t_3^{-1})\cdot
({\sigma_1\sigma_2^{-1}\sigma_1^{-1}} + {\sigma_1^{-1}\sigma_2\sigma_1})
\\
&+ (t_1t_3^{-1}-t_1^{-1}t_3)\cdot
({\sigma_1\sigma_2\sigma_1} + {\sigma_1^{-1}\sigma_2^{-1}\sigma_1^{-1}}) ;
\end{aligned}
\\
\text{\rm(III$_\text{4B}$)}:=
{\sigma_1^{-1}\sigma_2\sigma_1^{-1}} - {\sigma_1\sigma_2^{-1}\sigma_1}
+ {\sigma_2\sigma_1^{-1}\sigma_2} - {\sigma_2^{-1}\sigma_1\sigma_2^{-1}} ;
\\
\text{\rm(III$_\text{8B}$)}:=
\begin{aligned}[t]
&t_1t_2\cdot{\sigma_1^{-1}\sigma_2^{-1}\sigma_1} - t_1^{-1}t_2^{-1}\cdot{\sigma_1\sigma_2\sigma_1^{-1}}
+t_2^{-1}t_3^{-1}\cdot{\sigma_1\sigma_2^{-1}\sigma_1^{-1}} - t_2t_3\cdot{\sigma_1^{-1}\sigma_2\sigma_1}
\\
&+t_1^{-1}t_3\cdot{\sigma_1\sigma_2\sigma_1} - t_1t_3^{-1}\cdot{\sigma_1^{-1}\sigma_2^{-1}\sigma_1^{-1}}
+{\sigma_1^{-1}\sigma_2\sigma_1^{-1}} - {\sigma_1\sigma_2^{-1}\sigma_1} .
\end{aligned}
\end{gather*}
}
\end{exam}

\subsection{Properties of skein relators}
\label{subsec:properties_of_relators}

\begin{prop}
\label{prop:sym_rel}
Assume that
\[
C_1\cdot\nabla_{L_{\beta_1}}+\dots+C_k\cdot\nabla_{L_{\beta_k}}=0
\]
is a skein relation. Then for any given braid $\alpha\in B_n$,
the following equations are also skein relations:
\begin{gather*}
C_1\cdot\nabla_{L_{(\beta_1\alpha)}}+\dots+C_k\cdot\nabla_{L_{(\beta_k\alpha)}}=0;
\\
{^\alpha\!}C_1\cdot\nabla_{L_{(\alpha\beta_1)}}+\dots+{^\alpha\!}C_k\cdot\nabla_{L_{(\alpha\beta_k)}}=0,
\end{gather*}
where each ${^\alpha\!}C_h(t_1,\dots,t_n)$
is obtained from $C_h$ by permuting variables,
replacing each variable $t_i$ of $C_h$ with the variable $t_j$ such that $j^\alpha=i$.
The permutation of variables $C\mapsto{^\alpha\!}C$ defines a left $B_n$-action
on the coefficient domain $\mathbb P_n$.
\end{prop}

\begin{proof}
(i) Look at the cylinder where the colored links $L_{(\beta_1\alpha)},\dots,L_{(\beta_k\alpha)}$
are represented differently by braids $\beta_1\alpha,\dots,\beta_k\alpha$, respectively.
In the upper half cylinder they are represented by braids $\beta_1,\dots,\beta_k$.
So the assumption implies the conclusion.

(ii) Look at the cylinder where the colored links $L_{(\alpha\beta_1)},\dots,L_{(\alpha\beta_k)}$
are represented differently by braids $\alpha\beta_1,\dots,\alpha\beta_k$, respectively.
In the lower half cylinder they are represented as braids $\beta_1,\dots,\beta_k$.
So the assumption implies a linear equality
$C_1'\cdot\nabla_{L_{(\alpha\beta_1)}}+\dots+C_k'\cdot\nabla_{L_{(\alpha\beta_k)}}=0$,
in which the coefficients $C_h'$ are essentially the same as $C_h$, but with
the variables $\{t_i\}$ permuted.
Note that by our convention, the variable $t_j$ in $C'_h$ refers to
the color label of the $j$-th strand at the top of the colored braid $(\alpha\beta_h)^{(t_1,\dots,t_n)}$,
i.e., the $j$-th strand at the top of the whole cylinder (for $\alpha\beta_h$).
This strand is exactly the $j^{\alpha}$-th strand at the top of the lower half cylinder (for $\beta_h$).
So the variable $t_j$ in $C'_h$ is the variable $t_{j^{\alpha}}$ in $C_h$.
Hence $C'_h={^\alpha\!}C_h$.

(iii) The action of $B_n$ on $\mathbb P_n$ is clearly a left action.
\end{proof}

\begin{defn}
\label{defn:twisted_group-algebra}
In $\mathbb P_nB_n$ define the multiplication by
\[
(C_1\beta_1)\cdot(C_2\beta_2) := (C_1\cdot{}^{\beta_1}C_2)(\beta_1\beta_2).
\]
Then $\mathbb P_nB_n$ becomes a $\mathbb P_n$-algebra.
Note that it is not the ordinary group-algebra, but is \emph{twisted\/}
by the left $B_n$-action on the coefficient domain $\mathbb P_n$.
\end{defn}

\begin{prop}
\label{prop:relator_ideal}
Properties of skein relators:

\textup{(1)} Skein relators are homogeneous elements of the algebra $\mathbb P_nB_n$.

\textup{(2)} The sum of two skein relators is not a skein relator
unless they have the same underlying permutation.

\textup{(3)} A skein relator when left- or right-multiplied by a braid is still a skein relator.

\textup{(4)} Linear combinations of skein relators form a two-sided ideal\/ $\mathfrak R_n$
(called the \emph{relator ideal\/})
in $\mathbb P_nB_n$.

\textup{(5)} Every homogeneous element of\/ $\mathfrak R_n$ is a skein relator.
\qed
\end{prop}

\begin{exam}
\label{exam:relation(III_4)&(III_8)}
By computation we can verify the following three equalities in $\mathbb P_nB_n$:
\begin{gather*}
(\text{III}_\text{4B})=
\frac{\sigma_1\cdot (\text{III}_\text{B})\cdot \sigma_2^{-1} - \sigma_1^{-1}\cdot (\text{III}_\text{B})\cdot \sigma_2}{t_1t_2-t_1^{-1}t_2^{-1}} ;
\\
\begin{aligned}
(\text{III}_\text{8B})&=
\frac{(\sigma_1^2-t_1t_2\cdot e)\cdot (\text{III}_\text{B})}{t_1t_2-t_1^{-1}t_2^{-1}}
+ (\text{II}_\text{B})\cdot {\sigma_1\sigma_2\sigma_1^{-1}}
\\
&- (\text{II}_\text{B})\cdot
\left[
\frac{t_2t_3-t_2^{-1}t_3^{-1}}{t_1t_2-t_1^{-1}t_2^{-1}}\cdot {\sigma_1\sigma_2^{-1}\sigma_1^{-1}}
+ \frac{t_1t_3^{-1}-t_1^{-1}t_3}{t_1t_2-t_1^{-1}t_2^{-1}}\cdot {\sigma_1\sigma_2\sigma_1}
\right] ;
\end{aligned}
\\
(\text{III}_\text{B})
=\sigma_1^{-1}\sigma_2^{-1}\sigma_1^{-1}\cdot(\text{III}_\text{8B})\cdot\sigma_1\sigma_2\sigma_1
-(\text{III}_\text{8B})
+(\text{III}_\text{4B}) .
\end{gather*}

By the first two equalities and Proposition~\ref{prop:relator_ideal},
$(\text{III}_\text{B})\in\mathfrak R_n$ implies
both $(\text{III}_\text{4B})$ and $(\text{III}_\text{8B})$ are in $\mathfrak R_n$, hence are skein relators.
So Theorem~\ref{thm:relation(III_6)} implies Corollary~\ref{cor:relation(III_4)&(III_8)}.

Conversely, by the third equality, both $(\text{III}_\text{4B}), (\text{III}_\text{8B}) \in \mathfrak R_n$
would imply $(\text{III}_\text{B})\in\mathfrak R_n$.
So Corollary~\ref{cor:relation(III_4)&(III_8)} is powerful enough to imply Theorem~\ref{thm:relation(III_6)}.
\end{exam}

\begin{exam}
\label{exam:(III_7)}
Murakami's Axiom (3) for CPF \cite[page~126]{M2} corresponds to the relator
\begin{align*}
\text{\rm(III$_\text{7B}$)}&:=
(t_1+t_1^{-1})(t_2-t_2^{-1})\cdot{\sigma_2\sigma_1^2\sigma_2}-(t_2-t_2^{-1})(t_3+t_3^{-1})\cdot{\sigma_1\sigma_2^2\sigma_1}
\\
&-(t_1^{-1}t_3-t_1t_3^{-1})\cdot({\sigma_1^2\sigma_2^2}+{\sigma_2^2\sigma_1^2})
+(t_1^{-1}t_2t_3-t_1t_2^{-1}t_3^{-1})(t_3+t_3^{-1})\cdot{\sigma_1^2}
\\
&-(t_1+t_1^{-1})(t_1t_2t_3^{-1}-t_1^{-1}t_2^{-1}t_3)\cdot{\sigma_2^2}
-(t_1^{-2}t_3^2-t_1^2t_3^{-2})\cdot{e}
\end{align*}
which involves $7$ braids and has $22$ terms when all brackets are expanded.

It is indeed in $\mathfrak R_n$ because both
$\text{\rm(III$_\text{B}$)}, \text{\rm(III$_\text{8B}$)} \in \mathfrak R_n$ and
\begin{align*}
\text{\rm(III$_\text{7B}$)}& \cdot \sigma_1^{-1}\sigma_2^{-1}\sigma_1^{-1}
= t_1^{-1}t_3\cdot \text{\rm(III$_\text{B}$)} - (t_1t_3^{-1}-t_1^{-1}t_3)\cdot \text{\rm(III$_\text{8B}$)}
\\
&+ \frac{(t_2t_3^{-1}-t_2^{-1}t_3)\cdot \text{\rm(III$_\text{B}$)}
-(t_1t_3^{-1}-t_1^{-1}t_3)\cdot \sigma_1^{-1}\cdot \text{\rm(III$_\text{B}$)}\cdot \sigma_2}
{t_1t_2-t_1^{-1}t_2^{-1}}.
\end{align*}
\end{exam}

\section{An algebraic reduction lemma}
\label{sec:key_lem}

\begin{defn}
\label{defn:equivalence}
Let $\mathfrak I_n$ be the two-sided ideal in $\mathbb P_n B_n$
generated by $\text{\rm(II$_\text{B}$)}$ and $\text{\rm(III$_{\text{B}}$)}$.
(When $n=2$ we ignore $\text{\rm(III$_{\text{B}}$)}$.)

Two homogeneous elements of the algebra $\mathbb P_n B_n$ are \emph {equivalent modulo $\mathfrak I_n$}
(denoted by $\sim$\,) if
\begin{itemize}
\item they have the same underlying permutation, and
\item their difference is in $\mathfrak I_n$;
in other words, their difference is a $\mathbb P_n$-linear combination of elements of the form
$\beta\cdot \text{\rm(II$_\text{B}$)} \cdot\beta'$ or $\beta\cdot \text{\rm(III$_{\text{B}}$)} \cdot\beta'$
with $\beta,\beta'\in B_n$, all terms of the linear combination having the same underlying permutation.
\end{itemize}
\end{defn}

For example, by conjugation in $B_n$ we have
$\sigma_i^2 +\sigma_i^{-2} -(t_it_{i+1}+t_i^{-1}t_{i+1}^{-1})\cdot e \sim 0$ and
$(t_{i}^{-1}t_{i+1}^{-1}-t_{i}t_{i+1})\cdot
(\sigma_i\sigma_{i+1}\sigma_i^{-1} + \sigma_i^{-1}\sigma_{i+1}^{-1}\sigma_i)
+ (t_{i+1}t_{i+2}-t_{i+1}^{-1}t_{i+2}^{-1})\cdot(\sigma_i^{-1}\sigma_{i+1}\sigma_i
+ \sigma_i\sigma_{i+1}^{-1}\sigma_i^{-1})
+ (t_it_{i+2}^{-1}-t_i^{-1}t_{i+2})
(\sigma_i\sigma_{i+1}\sigma_i + \sigma_{i}^{-1}\sigma_{i+1}^{-1}\sigma_i^{-1})
\sim 0$,
for any $i$.

The following lemma is inspired by the statement of \cite[Lemma on page~460]{M1}
(where the definition of equivalence was different and the proof was not very clear).
\begin{lem}
\label{lem:reduction}
Modulo $\mathfrak I_n$, every braid $\beta\in B_n$ is
equivalent to a $\mathbb P_n$-linear combination of braids of the form
$\alpha\sigma_{n-1}^k \gamma$ with $\alpha,\gamma\in B_{n-1}$ where $k$ is $0$, $\pm1$ or $2$.
\end{lem}

A braid $\beta\in B_n$ can be written as
\[
\beta=\beta_0\sigma_{n-1}^{k_1} \beta_1\sigma_{n-1}^{k_2} \dots \sigma_{n-1}^{k_r}\beta_r
\]
where $\beta_j\in B_{n-1}$ and $k_j\neq0$.
We allow that $\beta_0$ and $\beta_r$ be trivial,
but assume other $\beta_j$'s are nontrivial.
The number $r$ will be denoted as $r(\beta)$.

The lemma will be proved by an induction on the double index $(n,r)$.
Note that the lemma is trivial when $n=2$, or $r(\beta)\leq1$.

It is enough to consider the case $r=2$, because induction on $r$ works beyond $2$.
Indeed, if $r(\beta)>2$, let $\beta'=\beta_1\sigma_{n-1}^{k_2} \dots \sigma_{n-1}^{k_r}\beta_r$, then
$r(\beta')<r(\beta)$. By inductive hypothesis $\beta'$ is equivalent to a
linear combination of elements of the form $\alpha'\sigma_{n-1}^{k'} \gamma'$,
hence $\beta$ is equivalent to a linear combination of elements of the form
$\beta_0\sigma_{n-1}^{k_1} \alpha'\sigma_{n-1}^{k'} \gamma'$.
This brings the problem back to the $r=2$ case.
Henceforth we assume $r=2$.

Since the initial and terminal part of $\beta$, namely $\beta_0$ and $\beta_r$,
do not affect the conclusion of the lemma, we can drop them.
So we assume $\beta=\sigma_{n-1}^{k_1} \beta_1\sigma_{n-1}^{k_2}$, where $\beta_1\in B_{n-1}$.

By the induction hypothesis on $n$, $\beta_1\in B_{n-1}$ is a linear combination of
elements of the form $\alpha_1\sigma_{n-2}^\ell \gamma_1$.
Note that $\alpha_1,\gamma_1\in B_{n-2}$ commute with $\sigma_{n-1}$.
So it suffices to focus on braids of the form
$\beta=\sigma_{n-1}^{k}\sigma_{n-2}^{\ell}\sigma_{n-1}^{m}$.

For the sole purpose of controlling the length of displayed formulas,
we assume $n=3$ below.
The proof for a general $n$ can be obtained by a simple change of subscripts,
replacing $\sigma_1,\sigma_2$ with $\sigma_{n-2},\sigma_{n-1}$ and
replacing $t_1,t_2,t_3$ with $t_{n-2},t_{n-1},t_{n}$, respectively.

Thus, Lemma~\ref{lem:reduction} has been reduced to the following
\begin{lem}
\label{lem:key_reduction}
Every $\sigma_2^{k}\sigma_1^{\ell}\sigma_2^{m}$ is equivalent \textup{(modulo $\mathfrak I_n$)} to a linear combination
of braids of the form $\sigma_1^{k'}\sigma_2^{\ell'}\sigma_1^{m'}$ where $\ell'$ is $0$, $\pm1$ or $2$.
\end{lem}

\begin{proof}
Modulo $\text{(II$_\text{B}$)}$, we may restrict the exponent
$k$ to take values $1$, $2$ and $3$ (we are done if $k$ is $0$).
If $k>1$ we can decrease $k$ by looking at
$\sigma_2^{k-1}(\sigma_2\sigma_1^{\ell}\sigma_2)$, so it suffices to prove the case $k=1$.
Again modulo $\text{(II$_\text{B}$)}$, we can restrict the exponents $\ell, m$ to the values $\pm1$ and $2$.
There are altogether 9 cases to verify.

{\it 5 trivial cases (braid identities) }:
\begin{alignat*}{3}
&\sigma_2\sigma_1\sigma_2=\sigma_1\sigma_2\sigma_1, \quad
&&\sigma_2\sigma_1\sigma_2^{-1}=\sigma_1^{-1}\sigma_2\sigma_1, \quad
&&\sigma_2\sigma_1^{-1}\sigma_2^{-1}=\sigma_1^{-1}\sigma_2^{-1}\sigma_1, \quad
\\
&\sigma_2\sigma_1\sigma_2^2=\sigma_1^2\sigma_2\sigma_1, \quad
&&\sigma_2\sigma_1^2\sigma_2^{-1}=\sigma_1^{-1}\sigma_2^2\sigma_1. \quad
\end{alignat*}

{\it The case $\sigma_2\sigma_1^{-1}\sigma_2$ }:
Multiplying $\text{\rm(III$_{\text{B}}$)}$ by $\sigma_2$ on the right and $\sigma_1^{-1}$ on the left,
and taking braid identities into account, we get the relation
\begin{align*}
(&t_1^{-1}t_2^{-1}-t_1t_2)\cdot (\sigma_2\sigma_1^{-1}\sigma_2
+ \sigma_1^{-1}\sigma_2\sigma_1^{-1})
+ (t_1t_3-t_1^{-1}t_3^{-1})\cdot
\\
&\cdot(\sigma_1^{-1}\sigma_2\sigma_1
+ \sigma_1\sigma_2^{-1}\sigma_1^{-1})
+ (t_2t_3^{-1}-t_2^{-1}t_3)\cdot(\sigma_1\sigma_2\sigma_1
+ \sigma_1^{-1}\sigma_2^{-1}\sigma_1^{-1})
\sim 0.
\end{align*}
The coefficient of $\sigma_2\sigma_1^{-1}\sigma_2$ is $t_1^{-1}t_2^{-1}-t_1t_2$
which is invertible in $\mathbb P_n$.
Divide through by this coefficient, then $\sigma_2\sigma_1^{-1}\sigma_2$ is equivalent to a linear combination of
braids of the form $\sigma_1^{\pm1}\sigma_2\sigma_1^{\pm1}$ and $\sigma_1^{\pm1}\sigma_2^{-1}\sigma_1^{-1}$.
So the case $\sigma_2\sigma_1^{-1}\sigma_2$ is verified.

{\it The case $\sigma_2\sigma_1^{-1}\sigma_2^2$ }:
Multiplying the previous relation by $\sigma_2$ on the right,
and taking braid identities into account, we see that
\begin{align*}
(&t_1^{-1}t_2^{-1}-t_1t_2)\cdot \left\{\sigma_2\sigma_1^{-1}\sigma_2^2
+ \sigma_1^{-1}(\sigma_2\sigma_1^{-1}\sigma_2)\right\}
+ (t_1t_3-t_1^{-1}t_3^{-1}) \cdot
\\
&\cdot (\sigma_2\sigma_1 + \sigma_1^2\sigma_2^{-1}\sigma_1^{-1})
+ (t_2t_3^{-1}-t_2^{-1}t_3)\cdot(\sigma_1^2\sigma_2\sigma_1
+ \sigma_2^{-1}\sigma_1^{-1})
\sim 0.
\end{align*}
Similarly to the above case, this reduces $\sigma_2\sigma_1^{-1}\sigma_2^2$
to the verified case $\sigma_2\sigma_1^{-1}\sigma_2$.

{\it The case $\sigma_2\sigma_1^2\sigma_2$ }:
Multiplying $\text{\rm(III$_{\text{B}}$)}$ on the right by $\sigma_1\sigma_2\sigma_1$, we get
\[
\begin{aligned}
(t_1^{-1}t_2^{-1}&-t_1t_2)\cdot(\sigma_1\sigma_2^2\sigma_1
+ \sigma_2^2)
+ (t_2t_3-t_2^{-1}t_3^{-1})\cdot (\sigma_2\sigma_1^2\sigma_2+ \sigma_1^2)
\\
&+ (t_1t_3^{-1}-t_1^{-1}t_3)\cdot(\sigma_1^2(\sigma_2\sigma_1^2\sigma_2) + e)
\sim 0.
\end{aligned}
\tag*{$\text{\rm(III$'$)}$}
\]
This is a relation among pure braids.
We have the following set of linear equations, where
the first line is $\text{\rm(III$'$)}$,
the second line is $\text{\rm(III$'$)}$ multiplied by $\sigma_1^{-2}$ on the left,
the third line is nothing but the relator $\text{\rm(II$_\text{B}$)}$,
and $K_1, K_2$ are linear combinations of braids of the form $\sigma_1^{k'}\sigma_2^{\ell'}\sigma_1^{m'}$.
\[
\left\{
\begin{aligned}
&(t_1t_3^{-1}-t_1^{-1}t_3)\cdot \sigma_1^2(\sigma_2\sigma_1^2\sigma_2)
+(t_2t_3-t_2^{-1}t_3^{-1})\cdot (\sigma_2\sigma_1^2\sigma_2)
\sim K_1,
\\
&(t_1t_3^{-1}-t_1^{-1}t_3)\cdot (\sigma_2\sigma_1^2\sigma_2)
+(t_2t_3-t_2^{-1}t_3^{-1})\cdot \sigma_1^{-2}(\sigma_2\sigma_1^2\sigma_2)
\sim K_2,
\\
&\sigma_1^2(\sigma_2\sigma_1^2\sigma_2)
-(t_1t_2+t_1^{-1}t_2^{-1})\cdot (\sigma_2\sigma_1^2\sigma_2)
+\sigma_1^{-2}(\sigma_2\sigma_1^2\sigma_2)
\sim 0.
\end{aligned}
\right.
\]
The determinant
\[
\det\begin{pmatrix}
t_1t_3^{-1}-t_1^{-1}t_3 & t_2t_3-t_2^{-1}t_3^{-1} & 0
\\
0 & t_1t_3^{-1}-t_1^{-1}t_3 & t_2t_3-t_2^{-1}t_3^{-1}
\\
1 & -t_1t_2-t_1^{-1}t_2^{-1} & 1
\end{pmatrix}
=(t_1t_2-t_1^{-1}t_2^{-1})^2
\]
is invertible in $\mathbb P_n$.
So solving these equations in $\mathbb P_n B_n$ we see that
$\sigma_2\sigma_1^2\sigma_2$ is equivalent to a linear combination of braids
of the form $\sigma_1^{k'}\sigma_2^{\ell'}\sigma_1^{m'}$, as desired.

{\it The case $\sigma_2\sigma_1^2\sigma_2^2$ }:
Multiplying $\text{\rm(III$'$)}$ by $\sigma_2$ on the right,
we get
\begin{align*}
(t_1^{-1}t_2^{-1}&-t_1t_2)\cdot
(\sigma_1^2\sigma_2\sigma_1^2 + \sigma_2^3 )
+ (t_2t_3-t_2^{-1}t_3^{-1})\cdot
(\sigma_2\sigma_1^2\sigma_2^2 + \sigma_1^2\sigma_2 )
\\
&+ (t_1t_3^{-1}-t_1^{-1}t_3)\cdot
(\sigma_1\sigma_2\sigma_1^3\sigma_2\sigma_1
+ \sigma_2 )
\sim 0.
\end{align*}
Since $\sigma_2\sigma_1^3\sigma_2^2$ reduces by $\text{\rm(II$_\text{B}$)}$ to the trivial case
$\sigma_2\sigma_1\sigma_2^2$ and the verified case $\sigma_2\sigma_1^{-1}\sigma_2^2$,
and the coefficient of $\sigma_2\sigma_1^2\sigma_2^2$ is invertible in $\mathbb P_n$,
the case $\sigma_2\sigma_1^2\sigma_2^2$ is also verified.

We have verified all 9 cases.
Modulo $\text{\rm(II$_\text{B}$)}$ we can assume $\ell'\in\{0,\pm1,2\}$.
Thus Lemma~\ref{lem:key_reduction} is proved.

The inductive proof of Lemma~\ref{lem:reduction} is now complete.
\end{proof}

The resulting $\mathbb P_n$-linear combination of braids of the form
$\alpha\sigma_{n-1}^k \gamma$ with $\alpha,\gamma\in B_{n-1}$ in the Lemma is not unique,
but the inductive proof gives us a recursive algorithm to find one.

\section{Proof of the Main Theorem}
\label{sec:proof_Main}

It is clear from Theorem~\ref{thm:relation(III_6)}
that the Conway potential function $\nabla_L$ satisfies all these axioms.
So, the existence of an invariant satisfying these axioms is already known.
The focus is the uniqueness.

In the rest of this section, we forget about the original definition of CPF, and
regard the symbol $\nabla$ as a well-defined invariant of colored links
which satisfies the five axioms listed in the statement of the Main Theorem.
We shall show that such a $\nabla$ is computable,
hence uniquely determined because we have assumed that $\nabla$ is well-defined.

Observe that for such a $\nabla$, everything in Subsections~\ref{subsec:skein_relators} and
\ref{subsec:properties_of_relators}
remains valid if we replace the ideal $\mathfrak R_n$ with the ideal $\mathfrak I_n$
defined in Definition~\ref{defn:equivalence},
and replace skein relators with homogeneous elements of $\mathfrak I_n$.

We will say a colored link $L$ is \emph{computable} if its $\nabla$ is computable.
By Proposition~\ref{prop:colored_Markov}, we may take $L$ to be the closure of a closable colored braid.
It suffices to prove the following claim for all $n$.

{\sc Inductive Claim($n$).}
For every closable colored $n$-braid $\beta$,
the closure $\wh \beta$ is computable.

The proof of Claim($n$) is by induction on $n$.
When $n=1$, Claim($1$) is true because there is only one $1$-braid,
with color symbol $t_1$.
Its closure is the trivial knot, whose $\nabla$ must be $(t_1-t_1^{-1})^{-1}$ by Axioms ($\Phi$) and (H).

Now assume inductively that Claim($n-1$) is true, we shall prove that Claim($n$) is also true.

Suppose $\beta$ is a closable colored $n$-braid.
By Lemma~\ref{lem:reduction}, the braid $\beta\in B_n$ is equivalent to (in a computable way)
a $\mathbb P_n$-linear combination of
braids of the form $\alpha\sigma_n^k \gamma$ with $\alpha,\gamma\in B_{n-1}$ and $k\in\{0,\pm1,2\}$.
Color these latter braids in the same way as for $\beta$.
Clearly they are all closable because they have the same underlying permutation.
By Example~\ref{exam:relator_vs_relation} the (mod $\mathfrak I_n$) equivalence preserves
the $\nabla$ of closure of closable colored braids.
So the computation of $\nabla_{\wh\beta}$ is reduces to the computation of
(a $\mathbb P_{\mathbf T}$-linear combination of) $\nabla_{\wh\beta'}$,
where $\beta'$ is of the form $\alpha\sigma_{n-1}^k \gamma$
with $\alpha,\gamma$ not involving $\sigma_{n-1}$ and $k\in\{0,\pm1,2\}$.

If $k=0$, the link $\wh\beta'$ has a free circle.
So its $\nabla$ must be $0$ by axiom (IO).

If $k=2$, the link $\wh\beta'$ is the
link $\wh {\alpha\gamma}$ (regarded as a closed $(n-1)$-braid)
with a ring attached to the last strand.
The latter link is computable by the inductive hypothesis Claim($n-1$),
so the former is also computable by Axiom ($\Phi$).

If $k=\pm1$, by Proposition~\ref{prop:colored_Markov} the link $\wh\beta'$ is isotopic to
the link $\wh {\alpha\gamma}$ (regarded as a closed $(n-1)$-braid)
which is computable by the inductive hypothesis Claim($n-1$).
So the former is also computable.

Thus Claim($n$) is proved.

The induction on $n$ is now complete.
Hence $\nabla$ is computable for every closed colored braid.
\qed

\begin{rem}
The induction above, together with the reduction argument of Sections \ref{sec:key_lem},
provides a recursive algorithm for computing $\nabla_{\wh\beta}$.
\end{rem}

\begin{rem}
A remarkable feature of this algorithm is that
it never increases the number of components of links.
In fact, all the reductions in Section~\ref{sec:key_lem} are by
Axioms (II) and (III) which respect the components,
while in this Section, components could get removed
but never added, by Axioms (IO) and ($\Phi$).
So if we start off with a knot, we shall always get knots along the way,
the Axioms (IO) and ($\Phi$) becoming irrelevant.
Hence the Corollary for Knots stated in the Introduction.
\end{rem}

\begin{rem}
The above proof of the Main Theorem does not show that the ideal $\mathfrak I_n$
is equal to the relator ideal $\mathfrak R_n$.
Certainly $\mathfrak I_n \subset \mathfrak R_n$.
It would be interesting to know whether the two relators
$\text{\rm(II$_\text{B}$)}$ and $\text{\rm(III$_{\text{B}}$)}$
are sufficient to generate the relator ideal $\mathfrak R_n$ in $\mathbb P_n B_n$, for every $n$.
\end{rem}

\end{document}